\documentclass[a4paper,12pt]{article}
\usepackage[utf8]{inputenc}
\usepackage{subfigure}
\usepackage{amsmath, epsfig,amssymb,amsthm}
\usepackage[round]{natbib}
\newtheorem{thm}{Theorem}[section]

\newtheorem{cor}[thm]{Corollary}

\newtheorem{lem}[thm]{Lemma}

\begin{document}
\title{On the genus filtration of diagrams over two backbones}

\author{Benjamin M. M. Fu$^1$ \ Christian M. Reidys$^2$\footnotemark \\
Department of Mathematics and Computer Science,\\
 University of Southern Denmark,\\
 Campusvej 55, DK-5230 \\
Odense M, Denmark\\
Email$^1$: benjaminfmm@imada.sdu.dk\\
number$^1$: 45-40485667\\
Email$^{2*}$: duck@santafe.edu\\
number$^{2*}$: 45-24409251\\
Fax$^{2*}$: 45-65502325
}
\date{}
\maketitle
\footnotetext[2]{Corresponding author.}
\newpage

\begin{abstract}
In this paper we compute the bivariate generating function of 
$\gamma$-matchings over two backbones, filtered by the number of 
arcs and the topological genus.
$\gamma$-matchings over two backbones are chord-diagrams,
obtained via concatenation and nesting of irreducible shapes of 
topological genus $\le \gamma$. We show that the key information is 
contained in the polynomials counting these shapes and provide 
recursions that allow to compute the latter. In particular we give a
bijection between such irreducible shapes over one and two backbones.
We present two applications of our results. The first is concerned with 
RNA-RNA interaction structures, obtained from the $\gamma$-matchings 
via symbolic methods. We secondly show that, using analytic-combinatorial 
methods, the topological genus satisfies a central limit theorem.
\end{abstract}

\section{Introduction}
\label{S:introduction}
In this paper we study the generating function of diagrams over two backbones. 
These combinatorial structures are filtered by the number of arcs and also
carry a natural topological filtration induced by the topological genus of 
their associated surface without boundary. 
Diagrams over two backbones play a central role in the context of folding 
algorithms of RNA-RNA interaction structures \citep{Huang:2bb}, 
\textit{i.e.}~complexes formed by two distinct RNA molecules. The key 
point here is that natural interaction structures are composed by 
irreducible ``motifs" of small topological genus. It appears that therefore
topological filtration offer a natural way of classifying such molecules.

It has been shown in \citep{Huang:2bb} that for fixed topological genus, there 
exist only finitely many irreducible motifs, called irreducible shadows. 
This motivates the notion of $\gamma$-diagrams or $\gamma$-matchings, i.e.~diagrams 
over two backbones composed by nesting such irreducible shadows of genus $\le \gamma$.
The algorithmic relevance of this finiteness lies in the fact that theses shadows can 
be individually evaluated and measured. This allows to design RNA folding algorithms that 
go beyond associating a global penalty for crossing arcs, see for instance \citep{Reidys:11a},
where this has been implemented for diagrams over one backbone.

Our main result is the bivariate generating function of $\gamma$-matchings over two 
backbones, filtered by the number of arcs and topological genus, ${\bf Q}_{\gamma}(u,t)$, 
in Corollary~(\ref{C:gamma}). 
The latter is expressed as an algebraic expression involving the polynomials
of irreducible shadows and the generating function of $\gamma$-matchings over one 
backbone ${\bf H}_{\gamma}(u,t)$, computed in \citep{thomas1}.

We finally discuss the implications of our results for RNA interaction structures
\citep{Huang:10a,rip:09}. To this end we show how to derive the relevant 
generating functions
via symbolic methods. In other words the biologically relevant structures can
be constructed in a modular fashion \citep{rita}, 
resulting in a composition of power series. 
We furthermore present a central limit theorem that is a corollary of 
Theorem~(\ref{T:mainthm}) and obtained via singularity analysis of ${\bf Q}_{\gamma}(u,t)$
employing the quasi-powers theorem~\citep{Hkhwang}.

\section{Some basic facts}
\label{S:basic fact}
\subsection{Diagrams}
\label{S:diagrams}
A \emph{diagram} is a labelled graph over the vertex set $[n] = \{1, . . . , n\}$,
in which each vertex has degree $\leq 3$. It can be represented by drawing 
its vertices in a horizontal line and its edges $(i, j)$, where $i < j$, in 
the upper half-plane. A \emph{backbone} is a sequence of connected, consecutive 
integers contained in $[n]$. \emph{A diagram over $b$ backbones} is a diagram 
together with a partition of $[n]$ into the $b$ backbones.

An interval $[i, i +1]$ is called a \emph{gap} if there exists a pair of subsequent 
backbones $B_1$ and $B_2$ such that $i(j)$ is the rightmost(leftmost) vertex 
of $B_1(B_2)$. The vertex $i$ is referred to as \emph{cut vertex}.

We call backbone edges \emph{$B$-arcs} and any other edge simply an arc. We shall
distinguish \emph{exterior} and  \emph{interior} arcs, where the former connect different 
backbones, see Fig.~\ref{F:notation}. Diagrams over multiple backbones without 
exterior arcs are simply disjoint unions of diagrams over one backbone. 

\begin{figure}[ht]
\begin{center}
\includegraphics[width=\textwidth]{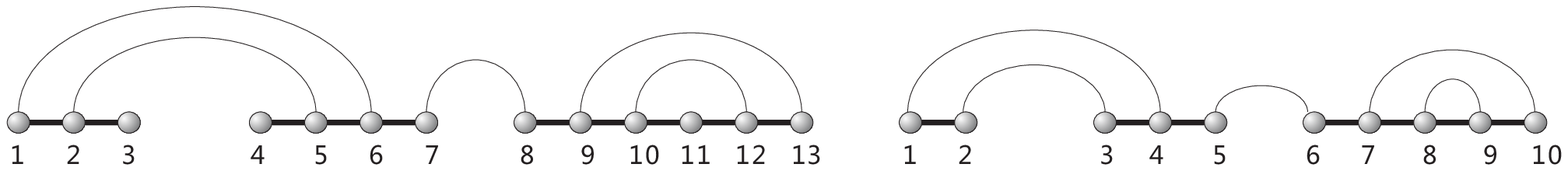}
\end{center}
\caption{\small LHS: a diagram over $[13]$ with arcs $\{(1,6), (2,5), (7,8), (9,13), (10,12)\}$ and $B$-arc
$\{(1,2), (2,3)\}$,$\{(4,5),(5,6),(6,7)\}$,$\{(8,9),(9,10),(10,11), (11,12),$\protect \\ $(12, 13)\}$,
 RHS: a matching derived by removing the isolated vertices and relabelling 
the vertices.}
\label{F:notation}
\end{figure}

The vertices and arcs of a diagram correspond to nucleotides and base pairs, 
respectively. For a diagram over $b$ backbones, the leftmost vertex of each 
backbone denotes the $5'$ end of the RNA sequence, while the rightmost vertex 
denotes the $3'$ end. A particular class of diagrams over two backbones
represents RNA interaction structures \citep{Huang:10a,rip:09}. 
Interaction structures are oftentimes represented alternatively by drawing 
the two backbones $R$ and $S$ on top of each other, where we label the vertices 
$R_1$ to be the $5'$ end of $R$ and $S_1$ to be the $3'$ of $S$. 

Let us next specify first properties of diagrams representing RNA interactions 
structures. A vertex $i$ is \emph{isolated} if it is not incident to any arc (except 
of backbone arcs). A diagram is \emph{connected} if and only if it is connected as a 
combinatorial graph (\textit{i.e.}~employing arcs as well as backbone arcs). 
A diagram that does not contain any isolated vertices is called a \emph{matching}.

An \emph{interior stack} of length $\tau$ is a maximal sequence of ``parallel" interior 
arcs, namely, $((i,j),(i+1,j-1),\cdots,(i+\tau -1,j-\tau +1 ))$. An interior stack 
is \emph{$\tau$-canonical} if it contains at least $\tau$ interior arcs. \emph{Exterior stacks} 
on $[i,j]$ and $\tau$-canonical exterior stacks are defined, accordingly.

A stack on $[i, j]$ of length $k$ naturally induces $(k-1)$ pairs of intervals 
of the form $([i+l, i+l+1],[j-l-1, j-l])$ where $0\leq l \leq k-2$. Any of these 
$2(k-1)$ intervals is referred to as a \emph{$P$-interval}. 
A \emph{$\tau$-canonical interaction structure} is a diagram in which
each stack has length at least $\tau$.
Any interval other than a 
gap or $P$-interval is called a \emph{$\sigma$-interval}. Clearly, a diagram over 
$[n]$, contains $(n-1)$ intervals and we distinguish three types: 
gap intervals, $P$-intervals and $\sigma$-intervals, see Fig.~\ref{F:backint}.
\begin{figure}[ht]
\begin{center}
\includegraphics[width=0.7\textwidth]{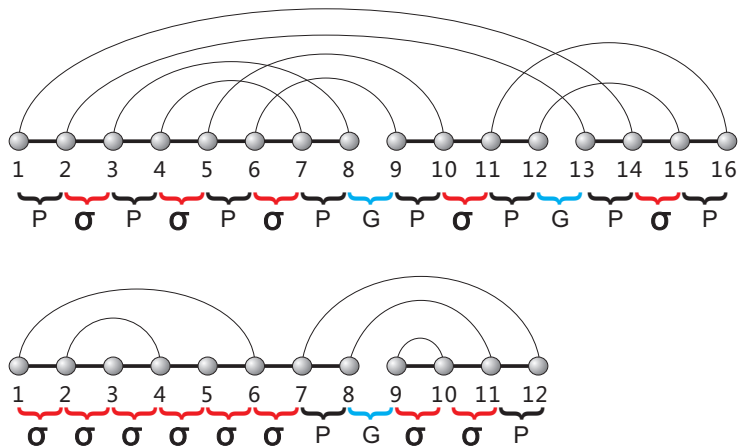}
\end{center}
\caption{\small Stacks and intervals:
gap intervals, $\sigma$-intervals and  $P$-intervals labelled  
by {\bf G}, {\bf $\sigma$} and {\bf P}. 
{\bf Top}: $4$ stacks; $SE_{1,14}^{2}$, $SI_{3,8}^{2}$, $SE_{5,10}^{2}$ and
$SE_{11,16}^{2}$. {\bf Bottom}: $4$ stacks; $SI_{1,6}^{1}$, $SI_{2,4}^{1}$, 
$SE_{7,12}^{2}$ and $SI_{9,10}^{1}$. Only $SE_{7,12}^{2}$ is 
$1$-canonical and $2$-canonical. All other stacks are 
$1$-canonical.
}
\label{F:backint}
\end{figure}
Let $\prec$ be the partial order on arcs given by $(i, j)\prec (i',j')$ if 
and only if $i'\leq i$ and $j \leq j'$.  Any diagram has a unique set of 
\emph{maximal arcs}. cf.~Fig.~\ref{F:maxi}.

\begin{figure}[ht]
\begin{center}
\includegraphics[width=0.7\textwidth]{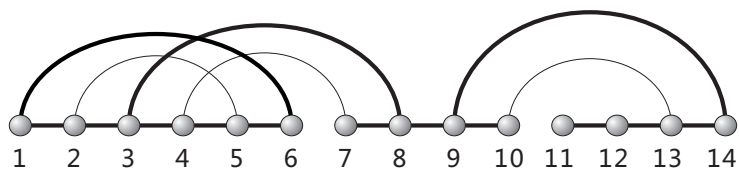}
\end{center}
\caption{\small Maximal arcs: $(1,6)$, $(3,8)$ and 
$(9,14)$ (bold).}
\label{F:maxi}
\end{figure}

\subsection{Diagrams to topological surfaces}
\label{S:surfaces}

The specific drawing of a diagram $G$ in the plane determines a cyclic ordering 
on the half edges of the underlying graph incident on each vertex, thus defining 
a corresponding \emph{fatgraph} $\mathbb{G}$. The collection of cyclic orderings is called 
\emph{fattening}, one such ordering on the half-edges incident on each vertex. Each 
fatgraph $\mathbb{G}$ determines an oriented surface $F(\mathbb{G})$, which is 
connected if $G$ is and has a topological genus $g(F(\mathbb{G}))$. Clearly,
$F(\mathbb{G})$ contains $G$ as a deformation retract and each $\mathbb{G}$ 
represents a cell-complex \citep{Massey:69} over $F(\mathbb{G})$.

A diagram $G$ hence determines a unique surface $F(\mathbb{G})$. 
Equivalence of simplicial and singular homology implies that
Euler characteristic, $\chi$, and genus, $g$, of $F(\mathbb{G})$ 
are independent of the choice of the cell-complex $\mathbb{G}$
and given by $\chi = v-e+r$ and $g=1-\frac{1}{2}\chi$, where 
$v,e,r$ are the number of discs, ribbons and boundary components 
in $\mathbb{G}$. 

Without affecting the topological type of the constructed surface, 
one may collapse each backbone to a single vertex with the induced 
fattening called \emph{the polygonal model of the RNA}. It is the orientation 
of each backbone from the $5'$ end to the $3'$ end that allows to 
transform the fatgraph of an RNA-structure or RNA-interaction structure into 
a fatgraph with one or two vertices. This backbone-collapse preserves 
orientation, Euler characteristic and genus, by construction. It is 
reversible by inflating each vertex to form a backbone. Using the 
collapsed fatgraph representation, we see that for a connected 
diagram over $b$ backbones, the genus $g$ of the surface is 
determined by the number $n$ of arcs and the number 
$r$ of boundary components, namely, $2-2g-r=v-e=b-n$.
 
Diagrams over one and two backbones are related by gluing, \textit{i.e.}, 
we have the mapping
\begin{displaymath}
 \alpha: \mathcal{E} \rightarrow \mathcal{D},
\end{displaymath}
where $\alpha(E)$ is obtained by keeping all arcs in $E$ and 
connecting the $3'$ end of $R$ and the $5'$ end of $S$. 
Furthermore, given two diagrams over two backbones, 
$E_1,E_2 \in \mathcal{E}$, we can 
insert $E_2$ into the gap of $E_1$ via concatenating 
the backbones $R_2$ and $R_1$ and 
$S_1$, $S_2$ preserving orientation. 
This composition is again a diagram over two backbones,
 $E_1\bullet E_2$, \textit{i.e.}~we have
\begin{displaymath}
 \mu: \mathcal{E} \times \mathcal{E} \rightarrow \mathcal{E},\quad
\mu (E_1,E_2)=E_1 \bullet E_2.
\end{displaymath}
It is straightforward to see that $\bullet$ is an associative 
product with unit given by the diagram over two empty backbones.
The product $\bullet$ is not commutative.

\subsection{Shadows}
A \emph{shadow} is a diagram with no non-crossing arcs or isolated vertices in
which each stack has size one. The shadow of a diagram is obtained by 
removing all non-crossing arcs, deleting all isolated vertices and collapsing 
each induced stack to a single arc. We shall denote the shadow of a diagram
 $X$ by $sd(X)$, note that $sd^2(X) =sd(X)$. Projecting into the shadow does 
not affect genus, \textit{i.e.}, $g(X) = g(sd(X))$. In case there 
are no crossing arcs, $sd(X)$ becomes an empty diagram on the same number 
of backbones as $X$. By definition, any empty backbone contributes one boundary
component. For example, for a diagram $X$ over $b$ backbones that contains
no crossing arcs, $sd(X)$ is a sequence of $b$ empty backbones with $b$ 
boundary components.

In the case of the shadows over two backbones, We distinguish the shadows by
type $A$ and type $B$. \emph{$A$-shadows} are those where both backbones are 
contained in one boundary component, all others are referred to as \emph{$B$-shadows}.
Let $\mathcal{A}_{g,m}$ denote the class of all the $A$-shadows of genus $g$ with 
$m$ arcs, and let $\mathcal{B}_{g,m}$ denote the class of all the $B$-shadows 
of genus $g$ with $m$ arcs.

Given a shadow over one backbone, we select any of its arcs, $a$. 
We inflate $a$ into a stack of size two and call the resulting diagram a 
\emph{$d$-shadow}. Let $\mathcal{I}_{g,m}$ be the class of shadows over one 
backbone having genus $g$ and $m$ arcs and $\mathcal{D}_{g,m+1}$ the class
of the $d$-shadows, having $(m+1)$ arcs, induced by $\mathcal{I}_{g,m}$.

\begin{lem}\label{L:bijection}
 There is a bijection 
\begin{displaymath}
 \alpha\colon \mathcal{A}_{g,m} \mathaccent\cdot\cup\; \mathcal{B}_{g-1,m} 
\longrightarrow \mathcal{I}_{g,m} \mathaccent\cdot\cup\; \mathcal{D}_{g,m}.
\end{displaymath}
\end{lem}
\begin{proof}
Given a genus $g$ shadow $s$ over two backbones we glue via $\alpha$ and 
mark the corresponding location where we glued. This generates either
a shadow over one backbone or a $d$-shadow with a mark, respectively.  
Furthermore, this operation is invertible. Namely, we can simply cut
the backbone at the marked point. 

It thus remains to consider the genera of the shadows involved.
Suppose first $s$ is a $A$-shadow, we will show that then $\alpha$ does
not change genus.
Indeed, gluing an $A$-shadow always splits a boundary 
component, whence the number of boundary components increases by one.
Evidently, the number of backbones decreases by one while the number of 
arcs does not change. Consequently, since $g'=(2+n-(r+1)-(b-1))/2=g$, 
the genus does not change.

Suppose next $s$ is a $B$-shadow. Then gluing will merge two boundary 
components. Thus, the number of boundary components decreases by one
and $g'=(2+n-(r-1)-(b-1))/2=g+1$ shows that the genus increases by one. 
As a result, 
\begin{displaymath}
\alpha\colon \mathcal{A}_{g,m} \mathaccent\cdot\cup\; \mathcal{B}_{g-1,m} 
\longrightarrow \mathcal{I}_{g,m} \mathaccent\cdot\cup\; \mathcal{D}_{g,m},
\end{displaymath}
is a bijection as stipulated.
\end{proof}

We furthermore have

\begin{thm}\label{T:finiteshadows}\citep{Huang:2bb}.
A shadow of genus $g\ge 0$ over two backbones has the following properties:\\
{\bf (a)} For $g\ge 1$ it contains at least $(2g+1)$ and at most $(6(g+1)-2)$ arcs; 
a shadow of genus $0$ has at least $2$ and at most $4$ arcs.
in particular, the set of such shadows is finite;\\
{\bf (b)} There exists at least one shadow over two
backbones with genus $g$ containing exactly $\ell$ arcs, where
\begin{equation}
 \ell =
\begin{cases}
(2g+1)\le \ell\le 6(g+1)-2 & \ \text{\rm for} \ g\ge 1, \\
2\le \ell \le 4            & \ \text{\rm for} \ g=0.
\end{cases}
\end{equation}
\end{thm}
\begin{proof}
First we recall an observation about shadows over one backbone \citep{Reidys:11a}.
shadows of genus $g\geq 1$ over one backbone have the following properties:

\emph{Claim $1$.}

{\bf (a)} A shadow of genus $g$ contains at least $2g$ and at most $(6g-2)$ 
arcs. In particular, for fixed $g$ there are only finitely many shadows;

{\bf (b)} For any $2g\le \ell\le 6g-2$, there exists a shadow of genus $g$ 
containing exactly $\ell$ arcs.

To prove this we note that if there is more than one boundary component, 
then there must be an arc with two different boundary components on its 
two sides. Removing this arc decreases $r$ by exactly one while preserving 
$g$ since the number of arcs is given by $n=2g+r-1$.
Furthermore, if there are $\nu_\ell$ boundary components of length
$\ell$ in the polygonal model, then $2n=\sum_\ell \ell\nu_\ell$ since
each side of each arc is traversed once by the boundary.  For a shadow,
$\nu_1=0$ by definition, and $\nu_2\leq 1$ as one sees directly.
Therefore $2n=\sum_\ell \ell\nu_\ell\geq 3(r-1)+2$, so
$2n=4g+2r-2\geq 3r-1$, \textit{i.e.}, $4g-1\geq r$.
Thus, we have $n=2g+(4g-1)-1=6g-2$, \textit{i.e.}, any shadow can contain at most
$(6g-2)$ arcs. The lower bound $2g$ follows directly from $n=2g+r-1$, since
$r\geq 1$.

Let $S_{2g}$ be a shadow containing $2g$ mutually crossing arcs,
\textit{i.e.}, each arc crosses any of the remaining $(2g-1)$ arcs. $S_{2g}$
has genus $g$ and contains a unique boundary component of length $4g$, 
\textit{i.e.}, traversing $4g$ non-backbone arcs counted with multiplicity.
We construct a new shadow $S_{2g+1}$ of genus $g$ containing $(2g+1)$ arcs,
by inserting an arc crossing into $S_{2g}$ from the $5'$ end of
$S_{2g}$ such that the boundary component in $S_{2g}$ splits into
one boundary component of length $3$ and another of length $4g+2-3=4g-1$.
The latter becomes the first boundary component of $S_{2g+1}$.
The newly inserted arc is by construction crossing, splits a boundary
component and preserves genus.
We now prove the assertion by induction of the number of inserted arcs.
By the induction hypothesis, there exists a shadow $S_{2g+i}$ of genus $g$ 
having $(2g+i)$ arcs, whose first boundary component has length $(4g-i)$.
Again, we insert a crossing arc as described above thereby splitting the 
first boundary component into one of length $3$ and the other of length 
$(4g-(i+1))$. After $i=4g-2$ such insertions, we arrive at a shadow 
whose first boundary component has length $2$ while all other boundary 
components have length $3$. Accordingly, there exists a set
$\{S_{2g},S_{2g+1},\ldots,S_{2g+(4g-2)}\}$ of shadows all having genus $g$, where
each $S_j$ contains $j$ arcs.

We finally observe that a shadow of genus $g=0$ over two backbones has at least 
$2$ arcs, while the maximum number of arcs contained in such a shadow is given by
$6(0+1)-2=4$.
For $g\ge 1$, it is impossible to cut a shadow of genus $g$ having $2g$ arcs and 
keep the genus. Thus the shadow of genus $g$ over two backbones has at least 
$(2g+1)$ arcs. By Lemma~(\ref{L:bijection}), We can always map an arbitrary shadow over 
two backbones of genus $g$ via $\alpha$ into a shadow over one backbone( of genus 
$g$ or $(g+1)$) or a $d$-shadow (of genus $g$). 
Claim $1$ guarantees that there are only finitely many such shadows 
and $d$-shadows, and the theorem follows.
\end{proof}
\subsection{Irreducibility}
\label{S:irreducible}
A diagram $E$ over $b$ backbones is called \emph{irreducible}, if it
is connected and for any two arcs, $\alpha_1,\alpha_k$, there exists a 
sequence of arcs 
\begin{displaymath}
 (\alpha_1,\alpha_2,\dots,\alpha_{k-1},\alpha_k),
\end{displaymath}
such that $(\alpha_i,\alpha_{i+1})$ are crossing. As proved in \citep{Huang:2bb}, 
we have the following corollary of Theorem (\ref{T:finiteshadows}).
\begin{cor}\label{C:finiteshadows}
An irreducible shadow having genus $g=0$ over two backbones contains at least
$2$ and at most $4$ arcs. For and $2\le \ell \le 4$, there exists an irreducible 
shadow of genus $g=0$ over two backbones having exactly $\ell$ arcs.
An irreducible shadow having genus $g\ge 1$ has the following properties:\\
{\bf (a)}
Every irreducible shadow with genus $g$ over two backbones contains at least
$(2g+1)$ and at most $(6(g+1)-2)$ arcs;\\
{\bf (b)}
For arbitrary genus $g$ and any $2g+1\le\ell\le 6g-2$, there exists an
\emph{irreducible} shadow of genus $g$ over one backbone having exactly $\ell$
arcs.
\end{cor}
Let $X$ be a diagram.  We call $S'$ an \emph{irreducible shadow} of $X$ (irreducible
$X$-shadow) if $S'$ is an irreducible shadow and any arc in $S'$ is
contained in $X$. 
$S'$ is a \emph{$(g,b,m)$-shadow} if $S'$ is a diagram over $b$ backbones having
genus $g$ and $m$ arcs. 
The set of irreducible $(g,b,m)$-shadows is denoted by $\mathcal{I}_{g,b,m}$.
Let $\mathcal{I}_{g,b}=\cup_{m} \mathcal{I}_{g,b,m}$.
 
According to Corollary~(\ref{C:finiteshadows}), the generating function 
$${\bf I}_{g,b}(u)=\sum\,{\bf i}_{g,b}(m)u^m$$ of the combinatorial class
$\mathcal{I}_{g,b}$ is in fact a polynomial. 

The generating polynomials for ${\bf I}_{g,b}(u)$ for $0\leq g\leq 1$ and 
$1\leq b\leq 2$ are
\begin{eqnarray*}
{\bf I}_{1,1}(u)&=& u^2  +  2u^3+   u^4,\\
{\bf I}_{2,1}(u)&=& 17 u^4+160 u^5+566 u^6+1004 u^7+961 u^8+476 u^9+96 u^{10},\\
{\bf I}_{0,2}(u)&=& 3u^2+3u^3+u^4,\\
{\bf I}_{1,2}(u)&=& 11 u^3+137 u^4+656 u^5+1520 u^6+1951 u^7+1436 u^8+572 u^9\\
\ & \  & +96 u^{10}.      
\end{eqnarray*}

A diagram is a \emph{$\gamma$-structure} if it is connected and all its 
irreducible shadows have genus at most $\gamma$. A \emph{$\gamma$-structure} 
is called $\tau$-canonical if every stack in the structure have at least $\tau$ arcs.
A \emph{$\gamma$-matching} is a \emph{$\gamma$-structure} without isolated vertices.
The combinatorial class of $\gamma$-matchings over one backbone is denoted 
by $\mathcal{H}_{\gamma}$ with generating function ${\bf H}_{\gamma}(u)$. We have 
\begin{thm}\label{T:H(u)}\citep{Reidys:gamma}:
Let $R=\mathbb{Z}[u]$. Then ${\bf H}_{\gamma}(u)$, satisfies
\begin{equation}\label{E:stru}
{\bf H}_{\gamma}(u)^{-1}
 = 1- \left(u\,{\bf H}_{\gamma}(u)
 + {\bf H}_{\gamma}^{-1}(u)\sum_{g\le \gamma}\,
 {\bf I}_{g,1}\left(\frac{u\,{\bf H}_{\gamma}^{2}(u)}
 {1-u \,{\bf H}_{\gamma}^{2}(u)}\right) \right).
\end{equation}
Furthermore, eq.~(\ref{E:stru}) determines ${\bf H}_{\gamma}(u)$ uniquely.
In case of $\gamma=1$, the coefficients of   ${\bf H}_{1}(u)$ are 
asymptotically given by
\begin{eqnarray}
[z^n]{\bf H}_{1}(u) & \sim &  k \, n^{-3/2}\, \left(\rho^{-1}\right)^n,
\end{eqnarray}
in which $k$ is some positive constant and $\rho^{-1}\approx 8.28425$.
\end{thm}
The combinatorial classes of $\gamma$-matchings over two backbones 
is denoted by  $\mathcal{Q}_{\gamma}$. 
We call $\gamma$-structures over two backbones also $\gamma$-interaction 
structures. Then 
\begin{thm}\citep{Qinjing}\label{T:tbgamma}
The generating function of $\gamma$-matchings over two backbones, 
${\bf Q}_{\gamma}(u)$, satisfies
\begin{equation}\label{E:tbgamma}
{\bf Q}_{\gamma}(u)=
\frac{{\bf H}^2_{\gamma}(u)\left(u{\bf H}^2_{\gamma}(u)+\sum_{g\leq \gamma}{\bf I}_{g,2}
\left(\frac{u{\bf H}^2_{\gamma}(u)}{1-u}\right)\right)}{1-u{\bf H}^2_{\gamma}(u)-
\sum_{g\leq \gamma}{\bf I}_{g,2}\left(\frac{u{\bf H}^2_{\gamma}(u)}{1-u}\right)}.
\end{equation}
For $\gamma= 0,1$ the coefficients of ${\bf Q}_{\gamma}(u)$ are asymptotically 
given by $[u^n]{\bf Q}_{\gamma}(u)\sim k_{\gamma}(\delta_{\gamma}^{-1})^n$ for some 
constant $k_{\gamma}>0$. In particular, 
$\delta_0^{-1}\approx 5.4252$ and $\delta_1^{-1}\approx 8.7266$.
\end{thm}
\section{Irreducible shadows and genus filtration} \label{S:Qmatching}
For shadows and matchings over one backbone, we have the bivariate generating function 
of irreducible shadows filtered by genus $g$ and arcs number $n$ denoted by 
\begin{equation}
{\bf I}(u,t)=\sum_{g\geq 1}{\bf I}_g(u)t^g=\sum_{g\geq 1}\sum_{n=2g}^{6g-2}i_g(n)u^nt^g.
\end{equation} 
We denote the class of all the matchings over $1$-backbone by $\mathcal{C}$.
Let furthermore $c_g(n)$ denote the number of matchings of genus $g$ with $n$ arcs
and
\begin{equation}
{\bf C}_g(u)=\sum_{m\geq 2g} c_g(n)u^n.
\end{equation}
Then the bivariate generating function of matchings filtered by genus $g$ and arc number $n$
is 
\begin{equation}
{\bf C}(u,t)=\sum_{g\geq 0}{\bf C}_g(u)t^g=\sum_{g\geq 0}\sum_{n\geq 2g}c_g(n)u^nt^g.
\end{equation}
In case of two backbones, By distinguishing $A$-shadows and $B$-shadows. 
We denote the bivariate generating polynomials, by
${\bf I}_{2,A}(u,t)$ and ${\bf I}_{2,B}(u,t)$:
\begin{align}
 & {\bf I}_{2,A}(u,t)=\sum_{g\geq 0}{\bf I}_{{2,A}_g}(u)t^g =
   \sum_{g\geq 1} \sum_{n=2g+1}^{6(g+1)-2}i_{{2,A}_g}(n)u^nt^g +
   \sum_{n=2}^{n=4}i_{{2,A}_0}(n)u^n, \\
 &{\bf I}_{2,B}(u,t)=\sum_{g\geq 0}{\bf I}_{{2,B}_g}(u)t^g =\sum_{g\geq 1}
 \sum_{n=2g+1}^{6(g+1)-2}i_{{2,B}_g}(n)u^nt^g +\sum_{n=2}^{n=4}i_{{2,B}_0}(n)u^n. 
\end{align} 
We furthermore denote the set of all the matchings over $2$-backbones by
$\mathcal{Q}$ and by $q_g(n)$ the number of matchings over two backbones of
genus $g$ with $n$ arcs. Then let
\begin{displaymath}
 {\bf Q}_g(u)=
\begin{cases}
\sum_{n\geq 2g+1} q_g(n)u^n  & \text{\rm for }  g\geq 1, \\
\sum_{n\geq 2} q_0(n)u^n     & \text{\rm for }  g=0.
\end{cases}
\end{displaymath}
The bivariate generating function of matchings over two backbones filtered by 
genus $g$ and arc number $n$ is ${\bf Q}(u,t)=\sum_{g\geq 0}{\bf Q}_g(u)t^g$.
The central observation is that ${\bf Q}(u,t)$ can be expressed via 
irreducible shadows as follows:

\begin{thm}\label{T:mainthm}
The generating functions ${\bf Q}(u,t)$,${\bf C}(u,t)$, ${\bf I}_{2,A}(u,t)$ and \\
${\bf I}_{2,B}(u,t)$ 
satisfy 
\begin{equation}\label{E:mainthm}
{\bf Q}(u,t)=\frac{{\bf C}(u,t)^2\left({\bf I}_{2,A}+{\bf I}_{2,B}-t{\bf I}_{2,B}^2-
 {\bf I}_{2,A}{\bf I}_{2,B}+u{\bf C}(u,t)^2(1-{\bf I}_{2,B})\right)}
{(1-t{\bf I}_{2,B})\left(1-u{\bf C}(u,t)^2-{\bf I}_{2,A}-t{\bf I}_{2,B}\right)},
\end{equation}
where ${\bf I}_{2,A}={\bf I}_{2,A}\left(\frac{u{\bf C}
(u,t)^2}{1-u{\bf C}(u,t)^2},t \right)$ 
and  
${\bf I}_{2,B}={\bf I}_{2,B}\left(\frac{u{\bf C}
(u,t)^2}{1-u{\bf C}(u,t)^2},t \right)$. 
\end{thm}
\begin{proof}
Let $s$ be an arbitrary $\mathcal{Q}$-matching. Consider the set of irreducible 
shadows which contain at least one exterior $s$-arc, $Sh_2(s)$. There exists exactly 
one element in $Sh_2(s)$, consisting of maximal arcs, \textit{i.e.}~all elements in 
$Sh_2(s)$ are nested via $\bullet$-product.

Our first goal shall be the computation of the generating function of two backbone 
matchings containing exactly $p$, distinct, nested shadows, ${\bf Q}_p(u,t)$.
 
{\emph{Claim $0$}.}  ${\bf Q}_0(u,t)=\frac{u{\bf C}(u,t)^4}{1-u{\bf C}(u,t)^2}$.

To prove Claim $0$, we note that here the two backbones are connected by at least 
$1$ exterior arc that does not belong to any irreducible $2$-shadow. 
If there are $l\geq 1$ such exterior arcs, these form an exterior stack. 
Let $s^*$ denote the particular $\mathcal{A}_l$-matching over $[2l]$ consisting 
of $l$ non-crossing exterior arcs. Any $\mathcal{A}_l$-matching, $s$,
can be obtained from $s^*$ in three steps, see Fig.~\ref{F:A3}:
\begin{figure}[!ht]
 \begin{center}
 \includegraphics[width=0.7\textwidth]{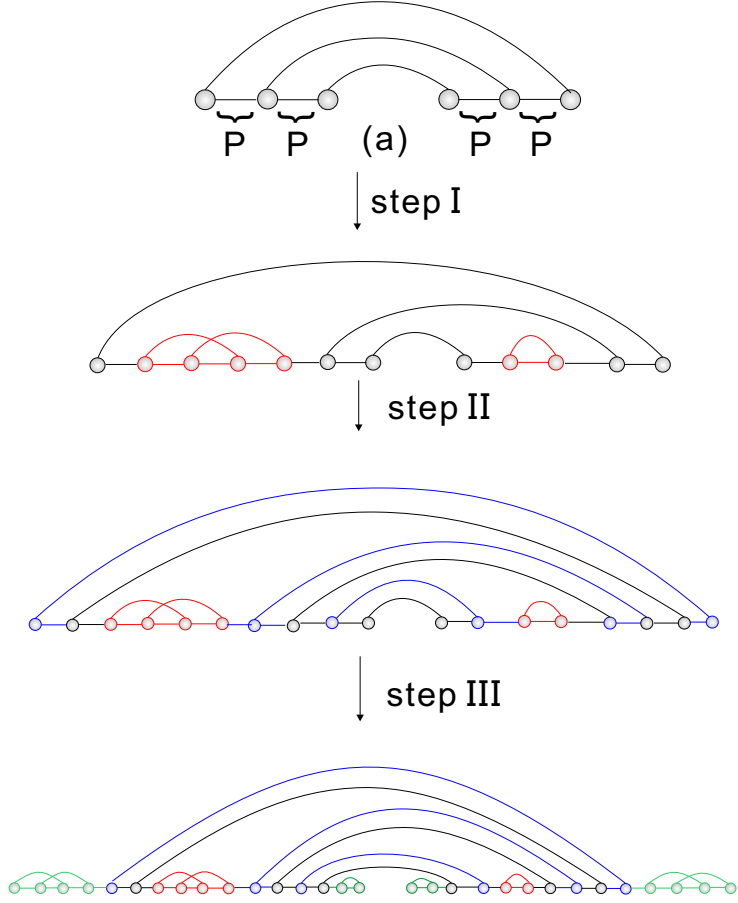}
 \end{center}
\caption{\small (a) is a $\mathcal{A}_3$ matching, 
consisting of $3$ exterior arcs.
Step {\bf I}: Insert non-empty matchings into the $P$ intervals. 
Step {\bf II}: Inflate each exterior arc to a stack. 
Step {\bf III}: Concatenation on each end of the (two) backbones.
}\label{F:A3}
 \end{figure}

({\bf I}) Insert non-empty $1$-backbone matchings into at least one of 
the two intervals of each pair of $([i,i+1],[2l-i,2l-i+1])$. 
Since there are $(l-1)$ such intervals in $s^*$, we obtain 
$\mathcal{U}^l \times (\mathcal{C}^2-\mathcal{E})^{l-1}$;

({\bf II}) Inflate each exterior arc in the derived matching into an exterior stack.
Since there are in total $l$ exterior arcs after Step (\textbf{I}), we obtain 
$\mathcal{N}_l=(\mathcal{U}\times SEQ(\mathcal{U}))^l 
\times (\mathcal{C}^2-\mathcal{E})^{l-1}$,
where $SEQ(\mathcal{U})=\mathcal{E}+\mathcal{U}+\mathcal{U}^2+\cdots$; 

({\bf III}) Concatenate each end of the (two) backbones with a (possible empty) 
$1$-backbone matching, \textit{i.e.}~$\mathcal{C}^4 \times \mathcal{N}_l$.

Accordingly, 
\begin{equation}
 \mathcal{A}_l=\mathcal{C}^4\times (\mathcal{U}\times SEQ (\mathcal{U}))^l \times 
(\mathcal{C}^2-\mathcal{E})^{l-1},
\end{equation}
for some $l\geq 1$. If the matching contains no $2$-backbone shadows, in step ({\bf I}) and 
({\bf III}) 
the genus increases by the genus of the added $1$-backbone matchings. In step ({\bf II}) the genus 
does not change. Then by summing over all $l\geq 1$:
\begin{equation}
 {\bf Q}_0(u,t)=\frac{u{\bf C}(u,t)^4}{1-u{\bf C}(u,t)^2},
\end{equation}
as claimed.

\emph{Claim $1$.}

\begin{align}
\nonumber {\bf Q}_{1}(u,t)= &{\bf I}_{2,A}\left(\frac{u{\bf C}
(u,t)^2}{1-u{\bf C}(u,t)^2},t \right)\frac{({\bf A}(u,t)+{\bf C}
(u,t)^2)^2}{{\bf C}(u,t)^2}\\
&+{\bf I}_{2,B}\left(\frac{u{\bf C}(u,t)^2}{1-u{\bf C}(u,t)^2},t
\right)\frac{t{\bf A}(u,t)^2+2t{\bf A}(u,t){\bf C}(u,t)^2 
+{\bf C}(u,t)^4}{{\bf C}(u,t)^2}.
\end{align}
Let $\phi$ be a fixed irreducible shadow of genus $g$, having $m$ arcs and let 
$\mathcal{W}_{\phi}$ be the set of matchings over two backbones, $s$, that contain 
only $\phi$ as shadow. In the following we shall construct such 
$\mathcal{W}_{\phi}$-matchings, see Fig.~\ref{F:1-structure}.
\begin{figure}[!ht]
 \begin{center}
 \includegraphics[width=0.7\textwidth]{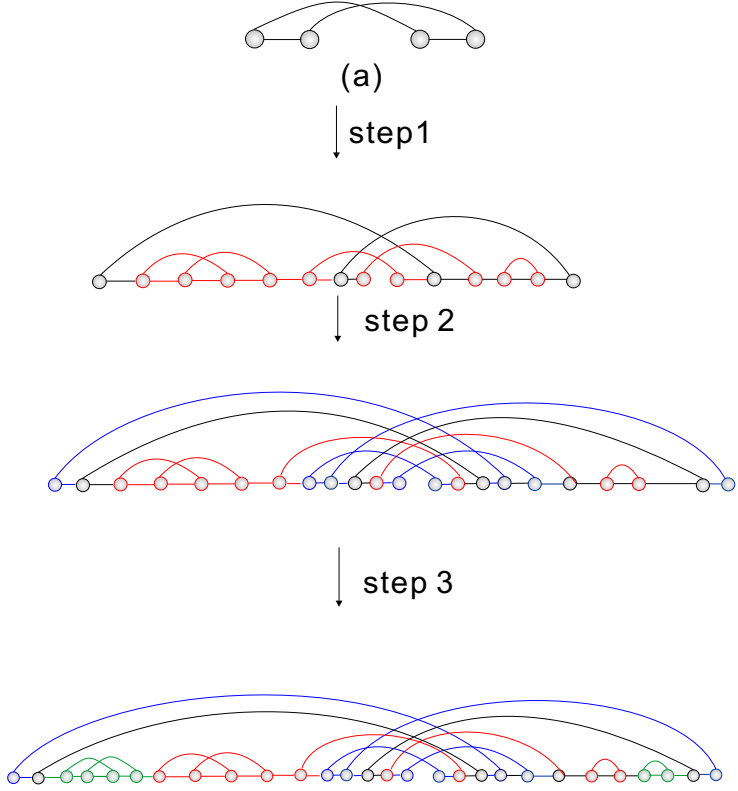}
 \end{center}
 \caption{\small (a) is an irreducible shadow of genus $0$ with $2$ arcs; 
step 1: Inflate each arc in the shadow to a sequence of induced exterior arcs (red arcs);
step 2: Inflate each exterior arc to a stack (blue arcs); 
step 3: Insert a $\mathcal{C}$-matching into the $2$ $\sigma$-intervals (green arcs).
 }\label{F:1-structure}
 \end{figure}
Let $\mathcal{AX}$ be the class of all possible diagrams such that 
$d \in \mathcal{AX}$ is either a $\mathcal{A}$-matching or a $\mathcal{X}$-matching,
where $\mathcal{X}$-matching is $\mathcal{C}\times \mathcal{C}$. That is,
$\mathcal{AX}=\mathcal{A}+\mathcal{X}$.

$\textbf{Step 1}:$ 
Inflate each arc in $\phi$ into a sequence of induced exterior arcs, \textit{i.e.}~an exterior 
arc together with at least one non-trivial matching over one backbone in either one or in
both $P$-intervals, \textit{i.e.}~
\begin{align}
\mathcal{N}=\mathcal{R} \times 
\left((\mathcal{C}-1)+(\mathcal{C}-1)+(\mathcal{C}-1)^2\right)=\mathcal{R} 
\times (\mathcal{C}^2-1).
\end{align}  
Clearly, we have ${\bf N}(u,t)=
u\left({\bf C}(u,t)^2-1\right)$(for a single induced exterior arc), 
Furthermore, for a sequence of induced arcs
$\mathcal{M}=SEQ(\mathcal{N})$, we have
\begin{align}
{\bf M}(u,t)=\frac{1}{1-u\left({\bf C}(u,t)^2-1\right)}.
\end{align}
Inflating each exterior arc into a sequence of induced arcs, $\mathcal{R}^m \times 
\mathcal{M}^m$, leads to
\begin{align}
u^m {\bf M}(u,t)^m=\left(\frac{u}{1-u({\bf C}(u,t)^2-1)}\right)^m.
\end{align}
\textbf{Step 2:} Inflate each exterior arc into a stack. The corresponding 
generating function is 
\begin{align}
\left(\frac{\frac{u}{1-u}}{1-\frac{u}{1-u}\left({\bf C}(u,t)^2-1\right)}\right)^m =
\left(\frac{u}{1-u {\bf C}(u,t)^2}\right)^m.
\end{align}
\textbf{Step 3:} Insert a $\mathcal{C}$-matching into the respective
$(2m-2)$ $\sigma$-intervals of $\phi$. The corresponding generating function is
${\bf C}(u,t)^{2m-2}$.

We thus arrive at
\begin{align}
t^g\left(\frac{u}{1-u{\bf C}(u,t)^2}\right)^m {\bf C}(u,t)^{2m-2}.
\end{align}
\begin{figure}[!ht]
\begin{center}
\includegraphics[width=0.75\textwidth]{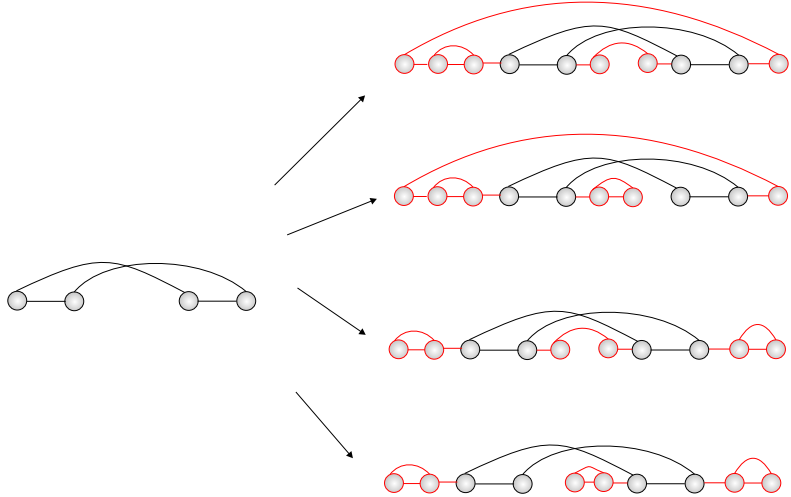}
\end{center}
\caption{\small The four cases of adding $\mathcal{AX}$.}\label{F:addAX}
\end{figure}
Next we add $\mathcal{AX}$ to the two termini, considering the four cases
$\mathcal{A}\bullet \phi \bullet \mathcal{A}$, 
$\mathcal{A}\bullet \phi \bullet \mathcal{X}$,
$\mathcal{X}\bullet \phi \bullet \mathcal{A}$,
and $\mathcal{X}\bullet \phi \bullet \mathcal{X}$,
see Fig.~\ref{F:addAX}.

\begin{figure}[!ht]
 \begin{center}
 \includegraphics[width=0.6\textwidth]{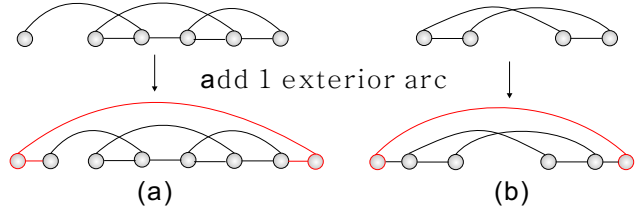}
 \end{center}
 \caption{\small $A$-shadow (a): if an exterior arc is added, the genus does not change. 
 $B$-shadow (b), if an exterior arc is added, the genus increases by $1$.
 }\label{F:addAX2}
 \end{figure}

Lemma 1 in \citep{Huang:2bb} allows to compute genus: if $\phi$ is a $A$-shadow,
the genus of the structure is just the sum of the substructure genera. 
In case of $\phi$ being a $B$-shadow, the genus contribution is $\left(g(\phi_B)+1\right)$
if there exist exterior arcs not contained in the shadow. 
However, if there are no such exterior arcs, the genus contribution is $g(\phi_B)$,
see Fig.~\ref{F:addAX2}.

Accordingly, the generating function of $\mathcal{W}_{\phi_{A}}$ is 
\begin{align}
{\bf W}_{\phi_{A}}(u,t)= t^g\left(\frac{u}{1-u {\bf C}(u,t)^2}\right)^m 
{\bf C}(u,t)^{2m-2} \left({\bf A}(u,t)+{\bf C}(u,t)^2\right)^2.
\end{align}
For $\phi$ being of type $B$ we obtain

\begin{equation}
 \begin{split}
{\bf W}_{\phi_{B}}(u,t)= & t^{g}\left(\frac{u}{1-u {\bf C}(u,t)^2}\right)^m 
{\bf C}(u,t)^{2m-2} \left(t{\bf A}(u,t)^2+ 2t {\bf A}(u,t){\bf C}(u,t)^2 \right. \\
& +\left. {\bf C}(u,t)^4\right).
\end{split}
\end{equation}

Our above arguments only depend on the number of $\phi$-arcs and the type of $\phi$.
Thus we have for any other irreducible shadow $\varrho$ of genus $g$ over two backbones 
having the same number of arcs and the same type,
${\bf W}_{\phi}(u,t)= {\bf W}_{\varrho}(u,t)$.
Consequently we arrive at
\begin{align}
{\bf Q}_{1}(u,t)= 
\sum_{g\geq 0}\left(\sum_{\phi \in A} {\bf W}_{\phi_A}(u,t)+\sum_{\phi \in B} 
{\bf W}_{\phi_B}(u,t)\right),
\end{align}
and 

\begin{equation}
 \begin{split}
 {\bf Q}_{1}(u,t)= &\sum_{g,m}\Big[ i_{2,A_g}(m) t^g 
\left(\frac{u}{1-u{\bf C}(u,t)^2}\right)^m  {\bf C}(u,t)^{2m-2}
\left({\bf A}(u,t)+{\bf C}(u,t)^2\right)^2 \\ 
 & + \left. i_{2,B_g}(m)t^{g}
\left(\frac{u}{1-u {\bf C}(u,t)^2}\right)^m
{\bf C}(u,t)^{2m-2}[t{\bf A}(u,t)^2+\right. \\
&2t {\bf A}(u,t){\bf C}(u,t)^2 +{\bf C}(u,t)^4] \Big]\\
 = &{\bf I}_{2,A}\left(\frac{u{\bf C}(u,t)^2}
{1-u{\bf C}(u,t)^2},t \right)
\frac{({\bf A}(u,t)+{\bf C}(u,t)^2)^2}{{\bf C}(u,t)^2}\\
&+{\bf I}_{2,B}\left(\frac{u{\bf C}(u,t)^2}{1-u{\bf C}(u,t)^2},
t \right)\frac{t{\bf A}(u,t)^2+2t{\bf A}(u,t){\bf C}(u,t)^2
+{\bf C}(u,t)^4}{{\bf C}(u,t)^2},
 \end{split}
\end{equation}

as claimed.

Suppose next we have $p\geq 2$. We shall distinguish three scenarios: 

\begin{itemize}
\item $s$ contains exactly $m>0$ $A$-shadows and $n$ $B$-shadows, where
$m+n=p$. Then the genus contribution is:
\begin{equation}
 \sum_{1}^{m} g(A_i)+\sum_{1}^{n} (g(B_i)+1),
\end{equation}
\item  $s$ contains no $A$-shadows, but there exist some exterior arcs
that do not belong to  $B$-shadows. 
Then genus contribution is $\sum_{1}^{p} (g(B_i)+1)$,
\item The exterior arcs of the matching are exclusively generated 
by $B$-shadows, see Fig.~\ref{F:typeB}. Then the exterior arcs of
$s$ can be decomposed in to $p$ shadows of type $B$, each of these
has genus $g_j$. Their corresponding genus is $\sum_{1}^{p}(g_j+1)-1$.
\end{itemize}
\begin{figure}[!ht]
 \begin{center}
 \includegraphics[width=0.5\textwidth]{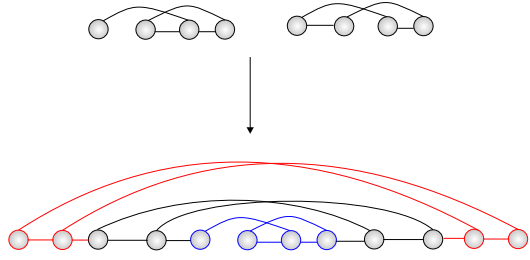}
 \end{center}
 \caption{\small A matching exclusively generated by $B$-shadows.  
 }\label{F:typeB}
 \end{figure}

Let $\mathcal{T}$ denote the combinatorial class of all the matchings over 
two backbones whose genus is given by $g=\sum g_A+\sum (g_B +1) + \sum g_C$,
where $\sum g_A$ is the sum of the genera of all $A$-shadows,
$\sum (g_B+1)$ is the sum of the genera of all $B$-shadows plus one, respectively.
 $\sum g_C$ is the sum of all the genera of inserted $\mathcal{C}$-matchings. That is
$\mathcal{T}$ neglects the genus decrease by $1$ when encountering the pure $B$-shadow 
case. Let $\mathcal{T}_p$ denote the class of these matchings that contain $p$ nested 
shadows and ${\bf T}_p(u,t)$ the corresponding generating function.

\emph{Claim $2.1$} 
\begin{equation}
\begin{split}
 \label{E:T1}
{\bf T}_1(u,t)= &\left[{\bf I}_{2,A}\left(\frac{u{\bf C}(u,t)^2}
{1-u{\bf C}(u,t)^2}, t\right) +
t \cdot{\bf I}_{2,B}\left(\frac{u {\bf C}(u,t)^2}
{1-u {\bf C}(u,t)^2}, t\right)\right]\\
& \cdot \frac{({\bf A}
(u,t)+{\bf C}(u,t)^2)^2}{{\bf C}(u,t)^2}.
\end{split}
\end{equation}
The proof is analogous to that of Claim $1$, the only difference  
emerging when adding the class $\mathcal{AX}$:
for $B$-shadows there is no need to distinguish the cases $\mathcal{A}$ 
and $\mathcal{X}$.

\emph{Claim $2.2$}
For $p\geq 2$, we have 
\begin{equation}
\begin{split}
 \label{E:Tp}
{\bf T}_{p}(u,t)= &
{\bf T}_{ p-1}(u,t)\cdot \left[{\bf I}_{2,A}
\left(\frac{u{\bf C}(u,t)^2}{1-u{\bf C}(u,t)^2},t\right)+ 
t \cdot {\bf I}_{2,B}\left(\frac{u{\bf C}(u,t)^2}
{1-u{\bf C}(u,t)^2},t\right)\right] \\
& \cdot\frac{\left({\bf A}(u,t)+{\bf C}(u,t)^2\right)}
{{\bf C}(u,t)^2}.
\end{split}
\end{equation}
To prove Claim $2.2$, let $\hat{\mathcal{T}}_{p}$ denote the subset of
 $\mathcal{T}$-matchings $s$ whose maximal arcs constitute an irreducible
shadow $sh$ in $Sh_2(s)$. Any $\mathcal{T}_p$-matching 
is of the form $\hat{\mathcal{T}}_{p}\bullet \mathcal{AX}$, \textit{i.e.}~
\begin{equation}
 \mathcal{T}_{p}=\hat{\mathcal{T}}_{p} \times (\mathcal{C}^2+\mathcal{A}).
\end{equation}
Let $\varrho$ be an irreducible shadow with genus $g$ and $m$ arcs and let 
${\hat{\mathcal{T}}}^{\varrho}_{p}$ denote the set of all
${\hat{\mathcal{T}}}_{p}$ matchings, $\hat{s}$, whose maximal arcs form $\varrho$. 
$\hat{s}$ can be obtained inductively by 
\begin{itemize} 
\item Nest exactly one $\mathcal{T}_{p-1}$-matching, $x$, into $\varrho$
 via the $\bullet$-product, if $\varrho$ is a $A$-shadow: 
 $g(x_1)=g(x)+g(\varrho)$ and $g(x_1)=g(x)+g(\varrho)+1$, 
 if $\varrho$ is a $B$-shadow;
\item Inflate each $\varrho$-arc into a sequence of induced arcs,
 $\nu \in  \mathcal{M}$, then we have $g(x_2)=g(x_1)+g(\nu)$;
\item Inflate each exterior arc in $x_2$ into a stack. There is no change
 in topological genus here, $g(x_3)=g(x_2)$;
\item Insert some $\mathcal{C}$-matchings, $c_1,\dots,c_k$  into the 
$(2m-2)$ $\sigma$-intervals of $\varrho$, then $g(\hat{s})=g(x_3)+\sum_j^kg(c_j)$.
\end{itemize}

Accordingly
\begin{equation}
{\hat{\mathcal{T}}}^{\varrho_A}_{p}=\mathcal{T}_{p-1} 
\times (\textbf{SEQ}(\mathcal{U})\times 
\textbf{SEQ}(\textbf{SEQ}(\mathcal{U})
\times (\mathcal{C}^2-1)))^m \times \mathcal{C}^{2m-2}\times \mathcal{V}^g,
\end{equation}
and
\begin{equation}
\hat{\mathcal{T}}^{\varrho_B}_{p}=\mathcal{T}_{p-1} 
\times (\textbf{SEQ}(\mathcal{U})\times 
\textbf{SEQ}(\textbf{SEQ}(\mathcal{U})\times 
(\mathcal{C}^2-1)))^m \times \mathcal{C}^{2m-2}\times \mathcal{V}^{g+1}.
\end{equation} 
Since the above constructions only depend on the number of arcs of $\varrho$,
and $\mathcal{T}_{p}=\hat{\mathcal{T}}_{p} \times (\mathcal{C}^2+\mathcal{A})$,
we have
\begin{equation}
 \begin{split}
 {\mathcal{T}}_{p}=&\mathcal{T}_{p-1} \times \Big(\sum_{g,m} 
i_{2,A_g}(m) \left(\textbf{SEQ}(\mathcal{U})\times 
\textbf{SEQ}(\textbf{SEQ}(\mathcal{U})\times 
(\mathcal{C}^2-1))\right)^m \\ & \times \mathcal{C}^{2m-2} 
 \times \mathcal{V}^g 
+ \sum_{g,m } i_{2,B_g}(m) \left(\textbf{SEQ}(\mathcal{U})\times 
\textbf{SEQ}(\textbf{SEQ}(\mathcal{U})\times 
(\mathcal{C}^2-1))\right)^m \\ & \times \mathcal{C}^{2m-2}
\times \mathcal{V}^{g+1}\Big) \times (\mathcal{C}^2+\mathcal{A}). 
\end{split}
\end{equation}

This implies
\begin{align}
\nonumber &{\bf T}_{p}(u,t)={\bf T}_{ p-1}(u,t)\cdot
 \left[{\bf I}_{2,A}\left(\frac{u{\bf C}(u,t)^2}{1-u{\bf C}(u,t)^2},
 t\right)+ t \cdot {\bf I}_{2,B}\left(\frac{u{\bf C}(u,t)^2}
 {1-u{\bf C}(u,t)^2},t\right)\right] \\
 &\cdot \left({\bf A}(u,t)+{\bf C}(u,t)^2\right)
 \cdot\frac{1}{{\bf C}(u,t)^2},
\end{align}
whence Claim $2.2$.

Claim $2.1$ and Claim $2.2$ allow us to recursively calculate ${\bf T}_p(u,t)$, for
$p \geq 2$. However, when calculating ${\bf T}_p(u,t)$, we neglected the 
genus-decrease in case of pure $B$-shadows. In order to correct this we introduce 
${\bf K}_p(u,t)$, i.e.~${\bf Q}_p(u,t)={\bf T}_p(u,t)+{\bf K}_p(u,t)$.

\emph{Claim $3$.}
\begin{align}
& {\bf K}_{1}(u,t)= 
\left[(1-t) \cdot  {\bf I}_{2,B}\left(\frac{u{\bf C}(u,t)^2}
{1-u{\bf C}(u,t)^2},t\right)\right] 
\cdot {\bf C}(u,t)^2, \\
&{\bf K}_{p}(u,t)=
{\bf K}_{p-1}(u,t)\cdot \left[t \cdot {\bf I}_{2,B}
\left(\frac{u{\bf C}(u,t)^2}{1-u{\bf C}(u,t)^2},
t\right)\right].
\end{align}
To prove the first equation we restrict Claim $1$ and Claim $2.1$ to the case of pure 
$B$-shadows.  As for the second equation, we restrict Claim $2.2$ to pure $B$-shadows and 
notice that following the proof of Claim $2.2$, ignoring the decrease of genus by one in 
case of pure $B$-shadows, does not affect recursion eq.~(\ref{E:Tp}). 
Therefore Claim $3$ follows which allows us obtain ${\bf K}_p(u,t)$ for all $p\geq 2$.
Consequently,
\begin{equation}
{\bf Q}(u,t)= 
{\bf Q}_{0}(u,t)+{\bf Q}_{1}(u,t) + \sum_{p\geq 2} 
\left({\bf T}_{p}(u,t) + {\bf K}_{p}(u,t) \right),
\end{equation}
where ${\bf Q}_0(u,t)$ and ${\bf Q}_1(u,t)$ follows from Claim 
$0$ and Claim $1$, ${\bf T}_p(u,t)$ via the Claims $2.1$ 
and $2.2$ and ${\bf K}_p(u,t)$ via Claim $3$.

Setting
\begin{center}
${\bf I}_{2,A}={\bf I}_{2,A}\left(\frac{u{\bf C}(u,t)^2}
{1-u{\bf C}(u,t)^2},t \right)\quad$
and 
$\quad{\bf I}_{2,B}={\bf I}_{2,B}\left(\frac{u{\bf C}(u,t)^2}
{1-u{\bf C}(u,t)^2},t \right)$,
\end{center}
we obtain
\begin{align}
{\bf Q}(u,t)=\frac{{\bf C}(u,t)^2\left({\bf I}_{2,A}
+{\bf I}_{2,B}-t{\bf I}_{2,B}^2-{\bf I}_{2,A}
{\bf I}_{2,B}+u{\bf C}(u,t)^2(1-{\bf I}_{2,B})\right)}
{(1-t{\bf I}_{2,B})(1-u{\bf C}(u,t)^2-{\bf I}_{2,A}
-t{\bf I}_{2,B})},
\end{align}
and the proof of the theorem is complete.
\end{proof}

\begin{cor}\label{C:recursion}
${\bf I}_{2,A_g}(y)$ and ${\bf I}_{2,B_g}(y)$ can be computed as follows:\\
$\bullet$ ${\bf I}_{2,A_0}(y)=0$,\\
$\bullet$ ${\bf I}_{2,A_{g+1}}(y)$ can be computed recursively via ${\bf I}_{g+1}(y)$, ${\bf I}_{2,A_i}(y)$,
${\bf I}_{2,B_i}(y)$, ${\bf Q}_i(y)$ and ${\bf C}_i(y)$, where $i\leq g$,\\
$\bullet$ 
\begin{equation}
{\bf I}_{2,B_g}(y)= 2(y^2+y)\frac{d {\bf I}_{g+1}(y)}{dy}
-{\bf I}_{g+1}(y)-{\bf I}_{2,A_{g+1}}(y).
\end{equation}
\end{cor}
\begin{proof} We prove the Corollary~(\ref{C:recursion}) in Section~(\ref{S:appendix}).
\end{proof}

In light of Corollary~(\ref{C:recursion}), it suffices to compute ${\bf C}_g(y)$, ${\bf I}_g(y)$ 
and ${\bf Q}_g(y)$.

As for ${\bf C}_g(y)$, suppose first $g=0$. Then ${\bf C}_0(y)$ is the 
generating function of the Catalan numbers, \textit{i.e.}~${\bf C}_0(y)=\frac{1-\sqrt{1-4y}}{2y}$.
For $g\ge 1$, ${\bf C}_g(y)=\sum c_g(n)y^n$ has been computed in \citep{harer} and 
\citep{Penner-orthogonal}. The key recursion discovered by \citep{harer} reads
\begin{lem}\citep{harer}
The $c_g(n)$ satisfy the recursion
\begin{equation}
(n+1)c_g(n)=2(2n-1)c_g(n-1)+(2n-1)(n-1)(2n-3)c_{g-1}(n-2),
\end{equation}
where $c_g(n)=0\ $ for $2g>n$.
\end{lem}

The polynomials ${\bf I}_{g}(y)$ have already been computed in \citep{Reidys:gamma}, the idea there
 is to construct an analogue of Corollary~(\ref{C:recursion}):
\begin{lem}
For $g\geq 1$, $I_g(y)$ satisfies the following recursion
\begin{equation}
\begin{split}
{\bf I}_g(y)=& {\bf C}_g(\theta(y)) -\theta(y)
\sum\limits_{i=0}^{g}{\bf C}_i(\theta(y)){\bf C}_{g-i}(\theta(y)) \\
&- \sum\limits_{j=1}^{g-1}[t^{g-j}]{\bf I}_j
\left(\frac{\theta(y)(\sum_{k=0}^{g-i}
{\bf C}_k(\theta(y))t^k)^2}{1-\theta(y)
(\sum_{k=0}^{g-i}{\bf C}_k(\theta(y))t^k)^2}\right),
\end{split}
\end{equation}
where $\theta(y)=\frac{y(y+1)}{(2y+1)^2}$.
\end{lem}

In particular,
\begin{align*}
& {\bf I}_1(y)=y^2(1+y)^2,\\
& {\bf I}_2(y)=y^4(1+y)^4(17+92 y+96 y^2),\\
& {\bf I}_3(y)=y^6(1+y)^6(1259+15928 y+61850 y^2+92736 y^3 + 47040 y^4).
\end{align*} 

${\bf Q}_g(y)$, the generating function of $2$-backbone matchings of genus $g$
has been computed in \citep{Hillary}. Here the authors established a bijection 
between unicellular maps \citep{chapuy} and bicellular maps. Their bijection has 
the following enumerative corollary

\begin{cor}
The generating function ${\bf Q}_g(y)$ and ${\bf C}_{g}(y)$ 
satisfy the following functional equation
\begin{equation}
 \sum^{g+1}_{g_1=0}{\bf C}_{g_1}(y){\bf C}_{g+1-g_1}(y)
+{\bf Q}_g(y)={\bf C}_{g+1}(y)/y,
\end{equation}
which is equivalent to the coefficient equation
\begin{equation}
 \sum_{g_1=0}^{g+1}\sum_{i\geq 0}^{n}c_{g_1}(i)
c_{g+1-g_1}(n-i)+q_g(n)=c_{g+1}(n+1).
\end{equation}
\end{cor}
Accordingly we derive
\begin{align*}
 {\bf I}_{2,A_0}(y)=& 0,\\
 {\bf I}_{2,A_1}(y)=& y^3(11+18y+8y^2),\\
 {\bf I}_{2,A_2}(y)=& y^5(y+1)(928+5378 y+12515 y^2
+14520 y^3+8376 y^4+1920 y^5),\\
{\bf I}_{2,A_3}(y)=& y^7(y+1)^2(162158+1835979 y
+8891051 y^2 + 24076143 y^3 \\
&+39943686 y^4 + 41655548 y^5 + 26715416 y^6 + 9649920 y^7\\
&+1505280 y^8),\\
 \textbf{I}_{2,B_0}(y)= & y^2(3+3y+y^2),\\
 \textbf{I}_{2,B_1}(y)= & y^4(y+1)(119+529 y+991 y^2
+960 y^3+476 y^4+96 y^5),\\
 {\bf I}_{2,B_2}(y)= & y^6(y+1)^2(13849+130518 y
+538113 y^2 + 1263849 y^3 \\
&+ 1847182 y^4 +1719618 y^5 + 995738 y^6 + 327936 y^7 \\
&+ 47040 y^8).
\end{align*}

\section{Two back-bones $\gamma$-matchings}
According to Corollary~(\ref{C:recursion}), we can recursively compute 
${\bf I}_{2,A_g}(u)$ and \\ ${\bf I}_{2,B_g}(u)$.
Let 
\begin{align}
 & {\bf I}_{{2,\gamma}_A}(u,t)=\sum_{g\leq \gamma} 
 {\bf I}_{2,A_g}(u) t^g, \\
 & {\bf I}_{{2,\gamma}_B}(u,t)=\sum_{g\leq \gamma}
  {\bf I}_{2,B_g}(u) t^g.
\end{align}  
Particularly,  for $\gamma =0$  and $1$, we have 
\begin{align*}
&{\bf I}_{{2,0}_A}(u,t)= 0,\\
&{\bf I}_{{2,0}_B}(u,t)= (3u^2+3u^3+u^4)t^0=3u^2+3u^3+u^4,\\
& {\bf I}_{{2,1}_A}(u,t)= (7u^3 + 6u^4)t, \\
&{\bf I}_{{2,1}_B}(u,t)=3u^2+3u^3+u^4 +(4u^3+131u^4+656u^5+1520u^6 \\
&+1951u^7+1436u^8+572u^9+96u^{10})t.
\end{align*}
Let furthermore $h_{\gamma}(g,n)$ denote the number of $\gamma$-matchings 
of genus $g$ with $n$ arcs. Then
\begin{align}
& {\bf H}_{\gamma} (u)=\sum_{g\leq \gamma}
\sum_n h_{\gamma}(g,n)u^n,  \quad {\bf H}_{\gamma}(u,t)=\sum_{g\leq \gamma}
\sum_n h_{\gamma}(g,n)u^n t^g.
\end{align}
${\bf H}_{\gamma}(u,t)$ has already been given in \citep{thomas1}.

Furthermore, let $q_{\gamma}(g,n)$ denote the number of
$\gamma$-interaction matchings of genus $g$ with $n$ arcs 
and
\begin{align}
& {\bf Q}_{\gamma} (u)=\sum_{g\leq \gamma}
\sum_n q_{\gamma}(g,n)u^n, \quad
 {\bf Q}_{\gamma}(u,t)=\sum_{g\leq \gamma}
\sum_n q_{\gamma}(g,n)u^n t^g.
\end{align}
We next compute the generating function $\gamma$-matchings over two backbones, ${\bf Q}_{\gamma}(u,t)$. We can see that
${\bf Q}_{\gamma}(u,t)$ and ${\bf Q}(u,t)$, discussed in Section~(\ref{S:Qmatching}), differ only
in terms of the range of the summation index of $g$. As a result, the proof of
Theorem~(\ref{T:mainthm}) can be duplicated and we derive
\begin{cor}\label{C:gamma}
The bivariate generating function of $\gamma$-matchings 
over two backbones: ${\bf Q}_{\gamma}(u,t)$, satisfies
\begin{align}
{\bf Q}_{\gamma}(u,t)=\frac{{\bf H}_{\gamma}(u,t)^2
\left({\bf {\bf I}}_{{2,\gamma}_A}+{\bf I}_{{2,\gamma}_B}-
t{\bf I}_{{2,\gamma}_B}^2-{\bf I}_{{2,\gamma}_A}{\bf I}_{{2,\gamma}_B}
+u{\bf H}_{\gamma}(u,t)^2(1-{\bf I}_{{2,\gamma}_B})
\right)}{(1-t{\bf I}_{{2,\gamma}_B})(1-u{\bf H}_{\gamma}
(u,t)^2-{\bf I}_{{2,\gamma}_A}-t{\bf I}_{{2,\gamma}_B})},
\end{align}
where
 \begin{math}
 {\bf I}_{{2,\gamma}_A}={\bf I}_{{2,\gamma}_A}
\left(\frac{u{\bf H}_{\gamma}(u,t)^2}
{1-u{\bf H}_{\gamma}(u,t)^2},t \right)
 \end{math}
and 
 \begin{math}
 {\bf I}_{{2,\gamma}_B}={\bf I}_{{2,\gamma}_B}
\left(\frac{u{\bf H}_{\gamma}(u,t)^2}
{1-u{\bf H}_{\gamma}(u,t)^2},t \right). 
 \end{math}
\end{cor}

\section{Discussion}

In this section we address $\gamma$-interaction structures and their genus distribution.
The passage from $\gamma$-matchings to $\gamma$-interaction structures employs the
notion of shapes. A matching $X$ is a \textit{shape} if each stack in $X$ is of length 
exactly one. Given an arbitrary matching $s$, its shape is obtained by collapsing each 
stack to a single arc and then removing any isolated vertices from the thus obtained 
diagram.

Let $\mathcal{S}_{\gamma}$ denote the set of shapes that are $\mathcal{Q}_{\gamma}$-matchings 
and let $\mathcal{Q}_{\gamma} (n,m)$ denote the combinatorial class of 
$\mathcal{Q}_{\gamma}$-matchings over $2n$ vertices with $m$ interior arcs of length $1$ 
(\emph{$1$-arcs}). 
Note that any $1$-arc is by definition an interior arc.
Furthermore, let $\mathcal{S}_{\gamma}(n,m,g)$ denote the class of all 
$\mathcal{S}_{\gamma}$-shapes over $2n$ vertices 
with $m$ $1$-arcs of genus $g$ with generating function ${\bf S}_{\gamma}(u,e,t)$. 
Since collapsing stacks, adding or deleting $1$-arcs do not affect genus, we can 
enrich the functional equation given in Lemma $6.1$ of \citep{Qinjing} by means of a genus
filtration:
\begin{equation}
{\bf S}_{\gamma}(u,t,e)=\frac{(1+u)^2}{(1+2u-ue)^2}
{\bf Q}_{\gamma}\left(\frac{u(1+u)}{(1+2u-ue)^2}, t\right).
\end{equation}
It is straightforward to obtain a $\tau$-canonical $\gamma$-interaction structure
from a shape by inserting isolated vertices and inflating arcs to stacks. 
All of these steps will not change the topological genus. Thus we can extend 
${\bf Q}_{\tau,\gamma}(z)$ of \citep{Qinjing} to a bivariate generating function. 
By symbolic methods, we eventually derive
\begin{thm}
Suppose $\gamma \geq 0$ and $\tau \geq 1$ and let 
$u_{\tau}(z)=\frac{(z^2)^{\tau-1}}{z^{2\tau}-z^2+1}$. 
Then the generating function of $\tau$-canonical $\gamma$ structures 
over two backbones is given by
\begin{equation}
 {\bf Q}_{\tau,\gamma}
(z,t)=\frac{1}{(u_{\tau}(z)z^2-z+1)^2}{\bf Q}_{\gamma}\left(\frac{u_{\tau}(z)z^2}
{(u_{\tau}(z)z^2-z+1)^2},t\right).
\end{equation}
\end{thm}

Let us secondly discuss the genus distribution of $\gamma$-interaction structures.
For this purpose, we study the random variable $X_{n,\tau,\gamma}$ having the distribution 
\begin{equation}
 P(X_{n,\tau,\gamma}=g)=\frac{{\bf Q}_{\tau,\gamma}(n,g)}
{{\bf Q}_{\tau,\gamma}(n)},
\end{equation}
where $g=0,1,\cdots,\lfloor\frac{n-1}{2}\rfloor$.

In case of $\gamma=0,1$, ${\bf Q}_{\gamma,\tau}(z,t)$ has the unique dominant singularity 
$\theta(t)$, where we compute a local, singular representation of the form
\begin{equation}
 {\bf Q}_{\gamma,\tau}(z,t)=g(z,t)+h(z,t)\left(1-\frac{z}{\theta(t)}\right)^{\alpha}
\end{equation}
for some real $\alpha \in \mathbb{R}-\mathbb{N}$ and functions $g(z,t)$, $h(z,t)\neq 0$ 
and $\theta(t)\neq 0$ that are analytic at $z =z_0>0$ and $t=1$. 
If $t$ is sufficiently close to $1$, there exists an analytic 
continuation of ${\bf Q}_{\gamma,\tau}(z,t)$ to the region $|z|<|\theta(t)|+\delta$, 
$|\text{arg}(z-\theta(t))|>\epsilon$ for some $\delta>0$ and $\epsilon>0$.

The two parameter version of the transfer lemma of\citep{Flajolet:07a}
in combination with the Quasi Power Theorem \citep{Hkhwang} implies
\begin{thm}
For $\gamma=0,1$ and $1\leq\tau\leq 10$, there exists a pair $(\mu_{\tau,\gamma},\sigma_{\tau,\gamma})$ 
such that 
the normalized random variable 
\begin{equation}
Y_{n,\tau,\gamma}=\frac{X_{n,\tau,\gamma}
-\mu_{\tau,\gamma}n}{\sqrt{n \sigma_{\tau,\gamma}^2}},
\end{equation}
converges in distribution to a Gaussian variable with $\mu_{\tau,\gamma}$ 
and $\sigma^2_{\tau,\gamma}$ are given by 
\begin{equation}
 \mu_{\tau,\gamma}=-\frac{\theta'(1)}{\theta(1)} \ \ \text{and} 
\ \ \sigma^2_{\tau,\gamma}=-\frac{\theta''(1)}{\theta(1)} -
\frac{\theta'(1)}{\theta(1)}+\left(\frac{\theta'(1)}{\theta(1)} \right)^2. 
\end{equation}
Furthermore there exist positive constants $c_1$,$c_2$,$c_3$ such that
\begin{equation}
 P\{||X_n-EX_n||\geq \epsilon \sqrt{n}\}\leq c_1e^{-c_2\epsilon^2},
\end{equation}
uniformly for $\epsilon\leq c_3\sqrt{n}$.
\end{thm}

In Table 1, we present the values of the pairs $(\mu_{\tau,\gamma},
\sigma_{\tau,\gamma})$ for $\gamma=0,1$.

\section{Appendix}\label{S:appendix}

\emph{Proof of Corollary (\ref{C:recursion})}.

\begin{proof}
\emph{Claim $1$.}
\begin{align}
\nonumber & \left({\bf I}_{2,A_g}(y)+
\sum\limits_{j=0}^{g-1}[t^{g-j}]{\bf I}_{2,A_j}
(\omega(\theta(y),t))\right)\left[{\bf C}_0^2(\theta(y))
(1-{\bf I}_{2,B_0}(y))+{\bf Q}_0(\theta(y))\right]\\ 
\nonumber &+\left({\bf I}_{2,B_g}(y)+
\sum\limits_{j=0}^{g-1}[t^{g-j}]{\bf I}_{2,B_j}
(\omega(\theta(y),t))\right)\left[{\bf C}_0^2(\theta(y))
-\theta(y){\bf C}_0^4(\theta(y))\right]\\
& = {\bf M}(\theta(y)).
\end{align} 
where $\omega(\theta(y),t)=
\frac{\theta(y){\bf C}^2(\theta(y),t)}{1-\theta(y){\bf C}^2(\theta(y),t)}$,
$\theta(y)=\frac{y(y+1)}{(2y+1)^2}$ and
${\bf M}(\theta(y))$ is a polynomial in the variable $y$.

To prove Claim $1$, we consider equation~(\ref{E:mainthm}) of 
Theorem~(\ref{T:mainthm}), multiply the denominator of the right hand side 
on both sides of it and then compute
\begin{align*}
&[t^g]{\bf Q}(u,t)+u[t^{g-1}]{\bf C}^2(u,t)
{\bf I}_{2,B}{\bf Q}(u,t)+[t^{g-1}]{\bf I}_{2,A}{\bf I}_{2,B}
{\bf Q}(u,t) \\ & +[t^{g-2}]{\bf I}_{2,B}^2{\bf Q}(u,t)
-u[t^g]{\bf C}^2(u,t){\bf Q}(u,t)-[t^g]{\bf I}_{2,A}
{\bf Q}(u,t)-2[t^{g-1}]{\bf I}_{2,B}{\bf Q}(u,t)\\
=&[t^g]({\bf I}_{2,A}+{\bf I}_{2,B}){\bf C}^2(u,t)
+u[t^g]{\bf C}^4(u,t)\\ 
&-[t^{g-1}]{\bf I}_{2,B}^2{\bf C}^2(u,t)
-[t^g]{\bf I}_{2,A}{\bf I}_{2,B}{\bf C}^2(u,t)-u[t^g]
{\bf I}_{2,B}{\bf C}^4(u,t).
\end{align*}
Notice that $[t^g]{\bf Q}(u,t)={\bf Q}_g(u)$, and
\begin{align*}
&[t^g]{\bf I}_{2,A}\left(\frac{u{\bf C}^2(u,t)}{1-u{\bf C}^2(u,t)},
t\right)=\sum_{j=0}^{g}[t^{g-j}]
{\bf I}_{2,A_j}\left(\frac{u{\bf C}^2(u,t)}
{1-u{\bf C}^2(u,t)}\right),\\
&[t^g]{\bf I}_{2,B}\left(\frac{u{\bf C}^2(u,t)}{1-u{\bf C}^2(u,t)},
t\right)=\sum_{j=0}^{g}[t^{g-j}]
{\bf I}_{2,B_j}\left(\frac{u{\bf C}^2(u,t)}{1-u{\bf C}^2(u,t)}\right).
\end{align*}
Let $\omega(u,t)=\frac{u{\bf C}^2(u,t)}{1-u{\bf C}^2(u,t)}$,
by the above we have now a recursion for ${\bf I}_{2,A_g}(u)$ 
and ${\bf I}_{2,B_g}(u)$. In view of ${\bf I}_{2,A_0}(u)=0$, we have

\begin{align*}
&[ t^g]{\bf I}_{2,A}(\omega(u,t),t)\left[{\bf C}_0(u)^2
-{\bf I}_{2,B_0}\left(\frac{u{\bf C}_0(u)^2}
{1-u{\bf C}_0(u)^2}\right)+{\bf Q}_0(u)\right] \\
& +[t^g]{\bf I}_{2,B}\left(\omega(u,t),t\right)
\left[{\bf C}_0(u)^2-u{\bf C}_0(u)^4\right] \\
&={\bf Q}_g(u)-u[t^g]{\bf C}(u,t)^2{\bf Q}(u,t)
-2[t^{g-1}]{\bf I}_{2,B}(\omega(u,t),t){\bf Q}(u,t)\\
&\ +u[t^{g-1}]{\bf I}_{2,B}(\omega(u,t),t){\bf C}(u,t)^2{\bf Q}(u,t)
+[t^{g-2}]{\bf I}_{2,B}(\omega(u,t),t)^2{\bf Q}(u,t)\\
&\ -\sum_{j=0}^{g-1}[t^j]{\bf I}_{2,A}(\omega(u,t),t)
[t^{g-j}]{\bf C}(u,t)^2
-\sum_{j=0}^{g-1}[t^j]{\bf I}_{2,B}(\omega(u,t),t)
[t^{g-j}]{\bf C}(u,t)^2\\
&\ +\sum_{\genfrac{}{}{0 pt}{}{g_1+g_2+g_3=g}{ g1\leq g-1,g2\leq g-1}}
[t^{g_1}]{\bf I}_{2,A}(\omega(u,t),t)[t^{g_2}]{\bf I}_{2,B}
(\omega(u,t),t)[t^{g_3}]{\bf C}(u,t)^2\\
&\ +u\sum_{j=0}^{g-1}[t^j]{\bf I}_{2,B}
(\omega(u,t),t)[t^{g-j}]{\bf C}(u,t)^4
+[t^{g-1}]{\bf I}_{2,B}(\omega(u,t),t){\bf I}_{2,A}
(\omega(u,t),t)\\
& \cdot {\bf Q}(u,t) \ +[t^{g-1}]{\bf I}_{2,B}(\omega(u,t),t)^2{\bf C}(u,t)^2
-u[t^g]{\bf C}(u,t)^4 \\
&\ -\sum_{j=0}^{g-1}[t^j]{\bf I}_{2,A}
(\omega(u,t),t)[t^{g-j}]{\bf Q}(u,t).
\end{align*}

We denote the right side of the above equation by ${\bf M}(u)$, 
and set $y=\frac{z{\bf C}^2_0(u)}{1-u{\bf C}_0^2(u)}$. Then
$u=\theta(y)=\frac{y(y+1)}{(2y+1)^2}$ and we derive
\begin{align}
\nonumber & \left({\bf I}_{2,A_g}(y)+
\sum\limits_{j=0}^{g-1}[t^{g-j}]{\bf I}_{2,A_j}
(\omega(\theta(y),t))\right)\left[{\bf C}_0^2
(\theta(y))(1-{\bf I}_{2,B_0}(y))+{\bf Q}_0(\theta(y))\right]\\ 
\nonumber &+\left({\bf I}_{2,B_g}(y)
+\sum\limits_{j=0}^{g-1}[t^{g-j}]{\bf I}_{2,B_j}
(\omega(\theta(y),t))\right)\left[{\bf C}_0^2(\theta(y))
-\theta(y){\bf C}_0^4(\theta(y))\right]\\
& = {\bf M}(\theta(y)),
\end{align} 
as claimed. Claim $1$ allows us to compute ${\bf I}_{2,A_{g+1}}(y)$ 
via  ${\bf I}_{g+1}(y)$, ${\bf I}_{2,A_i}(y)$,
${\bf I}_{2,B_i}(y)$, ${\bf Q}_i(y)$ 
and ${\bf C}_i(y)$, where $i\leq g$, as stipulated in 
Corollary~(\ref{C:recursion}).

\emph{Claim $2$.}
\begin{equation}
{\bf I}_{2,B_g}(y)+{\bf I}_{2,A_{g+1}}(y)=2(y^2+y)
\frac{d {\bf I}_{g+1}(y)}{dy} -{\bf I}_{g+1}(y).
\end{equation}
To prove Claim $2$, it suffices to show 
\begin{equation}\label{E:ABI}
 i_{2,A_g}(m) + i_{2,B_{g-1}}(m)=(2m-1)i_{g}(m)+ 2(m-1)i_{g}(m-1),
\end{equation}
where $i_{g}(m)$ is the number of irreducible shadows of genus $g$ having $m$ arcs 
over $1$ backbone. Furthermore, $i_{2,A_g}(m)$ and $i_{2,B_{g-1}}(m)$ denote the number 
of irreducible $A$- and $B$-shadows.

Close inspection of Lemma~(\ref{L:bijection}) shows that it induces a bijection
of irreducible shadows, where we define a $d$-shadow to be irreducible if it 
is induced by an irreducible shadow (by inflating an arc into a stack of size 
two). That is any irreducible $2$-backbone shadow can be obtained by cutting 
either an irreducible $1$-backbone shadow or an irreducible $d$-shadow.
This however implies recursion~(\ref{E:ABI}). To see this, we observe that an 
irreducible $1$-shadow with $m$ arcs has $(2m-1)$ cut-points and any
$s \in \mathcal{D}_{g,m}$ has two cut-points, the number of irreducible 
$d$-shadows with $m$ arcs being $(m-1)i_g(m-1)$. 

It thus remains to translate the recursion to an equation of generating functions:
\begin{equation}
{\bf I}_{2,B_g}(y)+{\bf I}_{2,A_{g+1}}(y)=2(y^2+y)
\frac{d {\bf I}_{g+1}(y)}{dy} -{\bf I}_{g+1}(y),
\end{equation} as claimed.
We computed ${\bf I}_{2,A_{g+1}}(y)$ in Claim $1$ and substitution yields
\begin{equation}
{\bf I}_{2,B_g}(y)= 2(y^2+y)\frac{d {\bf I}_{g+1}(y)}{dy}
-{\bf I}_{g+1}(y)-{\bf I}_{2,A_{g+1}}(y).
\end{equation}
There are no irreducible $A$-shadows of genus $0$, whence ${\bf I}_{2,A_0}(y)=0$. 
\end{proof}

\section*{Acknowledgements}

\label{sec:ack}
We want to thank Fenix Wen Da Huang and Thomas Jia Xian Li for discussions.

\section*{Author Disclosure Statement}

No competing financial interests exist.

\nocite{*}
\bibliographystyle{abbrvnat}

\label{sec:biblio}
\newpage

 \begin{table}[!ht]
 \begin{center}
 \begin{tabular}{c c c c c c c c c c c}
 \hline
& \multicolumn{2}{c}{$\tau=1$}&\multicolumn{2}{c}{$\tau=2$} &
\multicolumn{2}{c}{$\tau=3$} \\
\hline
& $\mu_{\tau,\gamma}$ & $\sigma^2_{\tau,\gamma}$& $\mu_{\tau,\gamma}$ 
& $\sigma^2_{\tau,\gamma}$& $\mu_{\tau,\gamma}$ & $\sigma^2_{\tau,\gamma}$ \\
 \hline
 $\gamma=0$&$0.065198$&$0.087702$&$0.029719$&$0.010128$&$0.019179$&$0.006550$\\
 $\gamma=1$&$0.091240$&$0.021067$&$0.041235$&$0.009358$&$0.026632$&$0.006043$\\
 \hline
 &\multicolumn{2}{c}{$\tau=4$}&\multicolumn{2}{c}{$\tau=5$} 
 &\multicolumn{2}{c}{$\tau=6$}  \\
 \hline
 & $\mu_{\tau,\gamma}$ & $\sigma^2_{\tau,\gamma}$& $\mu_{\tau,\gamma}$ 
 & $\sigma^2_{\tau,\gamma}$& $\mu_{\tau,\gamma}$ & $\sigma^2_{\tau,\gamma}$ \\
 \hline
 $\gamma=0$&$0.014168$&$0.004855$&$0.011245$&$0.003864$&$0.009331$&$0.003214$ \\
 $\gamma=1$&$0.019706$&$0.004481$&$0.015666$&$0.003571$&$0.013017$&$0.002974$\\
 \hline
 \end{tabular}
 
 \caption{Genus distribution: the central limit theorem for topological genus 
  in $\gamma$-interaction structures, for genus equals $0$ and $1$, and $1\leq \tau \leq 6$, we computed
  $\mu$ and $\sigma^2$ as in the table.  }
 
 \end{center}
 \end{table}  
 
\end{document}